\documentclass{amsart}%
\usepackage{amsmath}
\usepackage{amssymb}
\usepackage{amsfonts}
\usepackage{upref,amsthm,amsxtra,exscale}
\usepackage{cite}
\usepackage[colorlinks=true,urlcolor=blue,
citecolor=red,linkcolor=blue,linktocpage,pdfpagelabels,
bookmarksnumbered,bookmarksopen]{hyperref}
\usepackage{epsfig,graphics,color}
\usepackage{graphicx}%
\newtheorem{theorem}{Theorem}[section]
\theoremstyle{plain}

\newtheorem{lemma}[theorem]{Lemma}

\newtheorem{Remark}{Remark}

\numberwithin{equation}{section}

\begin{document}
\title[Non-cooperative Elliptic Systems]{Non-cooperative Systems Modeling Repulsive Interaction of Bose-Einstein Condensates in $\mathbb{R}^N$}
\author{Liliane A. Maia}
\address{Departamento de Matem\'{a}tica, UNB, 70910-900 Bras\'{\i}lia, Brazil.}
\email{lilimaia@unb.br}
\author{ Mayra Soares}
\address{Departamento de Matem\'{a}tica, UNB, 70910-900 Bras\'{\i}lia, Brazil.}
\email{ssc\_mayra@hotmail.com}
\author{Ricardo Ruviaro}
\address{Departamento de Matem\'{a}tica, UNB, 70910-900 Bras\'{\i}lia, Brazil.}
\email{ricardoruviaro@gmai.com}
\thanks{Research supported by FAPDF 0193.001300/2016, CNPq/PQ 308378/2017-2 (Brazil), and PROEX/CAPES, Brazil.}
\date{\today}

\begin{abstract}
Inspired by so many possible applications  of this class of problems, we seek solution for non-cooperative elliptic systems of two Schr\"odinger equations. General conditions are assumed under the potentials, which produces convenient spectral properties on the elliptic operator concerned, hence the non-cooperation characterizes it as a strongly indefinite problem. An Abstract Linking Theorem developed previously by the first two authors is the main tool, since variational approach is applied. Furthermore, super and asymptotically quadratic nonlinearities are considered. 
\end{abstract}
\maketitle

\section{Introduction}
\label{sec:introduction}

\hspace{0.75cm}This work is devoted to solve non-cooperative elliptic systems under hypotheses on the potentials introduced by B. Sirakov in \cite{Sir}. For $N\geq 2$ we consider the elliptic system given by

\begin{equation} \label{ES}
\left \{ \begin{array}{ll}

-\Delta u + V_1(x)u =  F_u(x,u,v) \quad \; \;\text{in} \quad \mathbb{R}^N\\
-\Delta v + V_2(x)v = - F_v(x,u,v) \quad \text{in} \quad \mathbb{R}^N,
\end{array}
\right.
\end{equation}
where, for $i=1,2$, the potentials $V_i \in L^\infty_{loc}(\mathbb{R}^N),$ satisfy\\

\hspace{-0.4cm}$(V_1)_i$ \ There exist constants $a_i\geq0$ such that
$V_i(x)\geq - a_i \quad \text{for \ all} \quad x \in \mathbb{R}^N;$\\

\hspace{-0.4cm}$(V_2)_i$ \ Defining for all $\Omega \subset \mathbb{R}^N$ open,
$$\nu_i(\Omega) := \inf\left\{\int_{\Omega}\left(|\nabla u(x)|^2+V_i(x)u^2(x)\right)\,dx : u \in H^1_0(\Omega), \ ||u||_{L^2(\Omega)}=1\right\},$$
assume that $\displaystyle\lim_{r\to +\infty}\nu_i(\mathbb{R}^N\setminus B_r)= +\infty$, where $B_r = \{x \in \mathbb{R}^N : |x|\leq r\} $; \\

\hspace{-0.4cm}$(V_3)_i$ \ Setting the first eigenvalue of operator $A_i:= -\Delta + V_i(x)$ in $L^2(\mathbb{R}^N)$ by
$$\lambda^i_1:= \inf\left\{\int_{\mathbb{R}^N}\left(|\nabla u(x)|^2+V_i(x)u^2(x)\right)\,dx : u \in H^1(\mathbb{R}^N), \ ||u||_{L^2(\mathbb{R}^N)}=1\right\},$$
assume that $\lambda^i_1 >0.$\\

\hspace{0.75cm}Note that classes of potentials satisfying conditions $(V_1)_i$ and $(V_2)_i$ are
	$$V_i(x) \to +\infty \quad \text{as} \quad |x|\to +\infty,$$
	or
	$$V_i(x) > V_0 >0 \quad \text{and} \quad \dfrac{1}{V_i(x)}\in L^1(\mathbb{R}^N),$$
	or even such that 
	$$|\Omega^{V_i}_M| <+\infty \quad \text{for \ all} \quad M>0,$$
	where $\Omega^{V_i}_M = \{x \in \mathbb{R}^N : V_i(x)<M\}$. \ However, \ potentials such as $V_i(x) = x_1^2x_2^2  \ .  .  . \  x_n^2 - a_i$, with constants $a_i>0$ such that $\lambda^i_1>0$, satisfy conditions $(V_1)_i-(V_3)_i$ although they do not satisfy any of the conditions mentioned above. Furthermore, the following hypothesis
	$$\lim_{r \to +\infty}\left|\bigcap_{i=1,2}\Big(\Omega^{V_i}_M \Big)\setminus (B_r)\right| = 0, \quad \text{for \ all} \quad M>0,$$
	is a sufficient condition for hypothesis $(V_2)_i$ (cf. \cite{Sir}, Theorem 1.4 and \cite{Ra}, Remark 1.3). 

\hspace{-0.4cm}$(V)$ \ There exists $h(x) \in L^\infty_{loc}(\mathbb{R}^N)$ and constants $\alpha>1$, $c_0>0$, $r_0>0$ such that
$$h(x)\leq c_0\left[1+\big(\max_{i=1,2}\{0,V_i(x)\}\big)^{\frac{1}{\alpha}}\right], \quad \text{if} \quad |x|\geq r_0.$$
Moreover, setting $a_0:=\displaystyle\inf_{x\in \mathbb{R}^N} h(x)$, assume that $a_0> {\lambda^1_1}>0$.\\

\hspace{0.75cm}An example which fulfils this hypothesis is given just before the main result of Section 2.

\hspace{0.75cm}Once again inspired by \cite{Sir}, suitable hypotheses involving function $h(x)$ are required on $F$ for the both asymptotically and super quadratic cases.

\hspace{0.75cm}A relevant motivation to study problem (\ref{ES}) stems from the system
$$
\left \{ \begin{array}{lll}
i \dfrac{\partial\Phi_1}{\partial t} = \Delta \Phi_1 - V_1(x)\Phi_1 - \mu_1|\Phi_1|^2\Phi_1 - \beta_{12}|\Phi_2|^2\Phi_1\\
\\
i \dfrac{\partial\Phi_2}{\partial t} = \Delta \Phi_2 - V_2(x)\Phi_2 - \mu_2|\Phi_2|^2\Phi_2 - \beta_{21}|\Phi_1|^2\Phi_2,
\end{array}
\right.
$$
which has been used to model Bose-Einstein condensates
in two different hyperfine spin states with the corresponding condensate wave functions $\Phi_j, \ j = 1,2$. In this case, $V_j$ is the magnetic trapping potential for the respective hyperfine spin state, the
constants $\mu_j$ and $\beta_{jk}$ are the intra-species and interspecies scattering lengths, which represent the interactions between particles. If $\beta_{12}>0$ it means there is repulsive interaction between particles $1$ and $2$, on the other hand, if $\beta_{12}<0$, there exists attractive interaction between them.
 In purpose of finding its solitary wave solutions of the form $\Phi_j = e^{-i\lambda_j t}u_j(x), \ \lambda_j>0,$ for $j=1,2$, the system above is transformed into the elliptic system
$$
\left \{ \begin{array}{ll}
 -\Delta u_1 + \big(V_1(x) + \lambda_1\big)u_1 + \mu_1u^3_1 + \beta_{12}u^2_2u_1 = 0\\
  -\Delta u_2 + \big(V_2(x) + \lambda_2\big)u_2 + \mu_2u^3_2 + \beta_{21}u^2_1u_2 = 0.
  \end{array}
  \right.
$$
Assuming $\beta_{12}\beta_{21} < 0$ this elliptic system is non-cooperative and therefore it is an practical example of our problem. For details of physical applications see \cite{CLLL,LW} and references therein.

\hspace{0.75cm} Originally, our motivation to study (\ref{ES}) came from \cite{Co,CM2}. Indeed, in \cite{Co} D. G. Costa treated the cooperative elliptic system 
$$
\left \{ \begin{array}{ll}
-\Delta u + a(x)u = F_u(x,u,v) \quad \text{in} \quad \mathbb{R}^N\\
-\Delta v + b(x)v = F_v(x,u,v) \quad \;  \text{in} \quad \mathbb{R}^N,
\end{array}
\right.
$$
under the hypotheses that $a(x), \ b(x) \to +\infty$ as $n \to +\infty$ and the nonlinear term $F$ is non-quadratic at infinity. Moreover, he mentioned that his results could be generalized to non-cooperative elliptic systems applying techniques found in \cite{CM2}. Following this spirit and trying to generalize the results in \cite{Co} for non-cooperative elliptic systems, we have profited from ideas in the remarkable paper \cite{Sir}, where B. Sirakov considered a scalar problem under more general hypotheses on the potentials and worked with a super-quadratic nonlinearity.

\hspace{0.75cm}To the best of our knowledge, non-cooperative elliptic systems under assumptions similar to those introduced by B. Sirakov have not been treated yet. In fact, most of the well known papers in non-cooperative elliptic systems deals with constant potentials or even considers only bounded domains, as for instance, \cite{CM2, DFD, Po, BC} and references therein. Although we are also able to find in the literature some relevant works on non-cooperative elliptic systems with non-constant potentials and treated on the whole space, normally \ their \ authors \ consider \ $V=V_i, \ i=1,2$, and assume hypotheses away from those required by us, check for example \cite{WXZ,FWXZ}.

\hspace{0.75cm}Due to hypotheses $(V_1)_i$ and $(
V_3)_i$, problem (\ref{ES}) must be tackled in the subspace of $H^1(\mathbb{R}^N,\mathbb{R}^2)$ defined by
\begin{equation}\label{E}
E:= \left\{(u_1,u_2)\in H^1(\mathbb{R}^N,\mathbb{R}^2) : \int_{\mathbb{R}^N} \big(V_1(x)u_1^2(x) + V_2(x)u_2^2(x)\big)\,dx < +\infty \right\},
\end{equation}
which is a Hilbert space continuously embedded in $H^1(\mathbb{R}^N,\mathbb{R}^2)$ when endowed with the inner product
\begin{equation}\label{ip}
\Big((u_1,u_2),(v_1,v_2)\Big)_E:=\sum_{i=1,2}\int_{\mathbb{R}^N}\big(\nabla u_i(x) \cdot \nabla v_i(x)+V_i(x)u_i(x)v_i(x)\big)\,dx,
\end{equation}
and the correspondent norm $\big|\big|(u_1,u_2)\big|\big|^2_E:=\Big((u_1,u_2),(u_1,u_2)\Big)_E$ (cf. \cite{Ra} and \cite{Sir}, both Lemma 2.1). Furthermore, $E$ is compactly embedded in the weighted spaces $L^s_{h(x)}(\mathbb{R}^N,\mathbb{R}^2)$, for $2\leq s <+\infty$ if $N=2$, and for $2 \leq s < 2^{\#} := 2^* - \dfrac{4}{\alpha(N-2)}$ if $N>2$ (cf. \cite{Ra}, Proposition 2.4 and also \cite{Sir}, Proposition 3.1).
Hence, the spectra of operators $A_i$ are discrete and consist of a sequence ${\lambda_n^i}$ of positive eigenvalues such that $\lambda^i_n \to +\infty$ as $n \to +\infty.$ (cf. \cite{Sir}, Lemma 3.2). Note that $2^\#$ is determined such that $2<2^\#\leq2^*$ taking into account that $\alpha>1$.

\hspace{0.75cm}In our point of view, the most interesting feature of hypotheses on $V_i$ is that the elliptic operator associated to problem (\ref{ES}) has pure point spectrum. Although such a spectrum has negative and positive parts composed by unbounded sequences, which makes the problem strongly indefinite, its discreteness brings the necessary compactness for the problem. Furthermore, the narrow relation between potentials $V_i$ and nonlinearity $F$, enables us to get the desired geometry for the problem. In view of these facts, we are able to find a non trivial solution for (\ref{ES}). 

\hspace{0.75cm}All hypotheses and observations made up to now, will be assumed throughout this paper. From now, asymptotically and super quadratic cases are going to be treated separately. Section 2 is dedicated to the asymptotically quadratic problem. In order to do so, we recall the Abstract Linking Theorem in \cite{MS}, which is used to prove the existence of solution in this case, as well as in the super-quadratic case. First, exploiting the profile of spectrum presented by the associated operator we decompose space $E$ in appropriate subspaces for the linking structure. In this environment, the requirements of the abstract result are verified: compactness, linking geometry and boundedness of Cerami sequences of the functional associated to the problem. Throughout Section 3, we follow the same script for the super quadratic case, under the necessary changes. In this case, it is worth to focus in the argument for the boundedness of the sequences. Here, the well known Ambrosetti-Rabinowitz condition \cite{AR} would not be sufficient as in the scalar problem. So, to circumvent this difficulty we betake the resourceful hypothesis $(f5)$ introduced by L. Jeanjean and K. Tanaka in \cite{JeTa}. Finally, we point out that this scheme can be applied to systems with more than two equations.

\section{Asymptotically Quadratic Elliptic Systems}

\hspace{0.75cm}With the purpose of studying the asymptotically quadratic case the following assumptions on the Hamiltonian $F$ are required.\\

\hspace{-0.6cm}$(F_1)$ \ $F(x,s,t) \in C^1(\mathbb{R}^N\times\mathbb{R}^2,\mathbb{R}^+)$  and
$$\dfrac{F(x,s,t)}{h(x)} = o(|(s,t)|^2), \quad \text{as} \quad |(s,t)|\to 0,$$
$$\dfrac{F(x,s,t)}{|(s,t)|^2} = h(x), \quad \text{as} \quad |(s,t)|\to +\infty,$$
uniformly in $x\in \mathbb{R}^N$.
\\

\hspace{-0.6cm}$(F_2)$ Setting $Q(x,s,t) := F_s(x,s,t)s + F_t(x,s,t)t - 2F(x,s,t)$,  for $(x,s,t) \in \mathbb{R}^N\times\mathbb{R}^2$, one has
$$\lim_{|(s,t)| \to +\infty} Q(x,s,t) = +\infty,$$
uniformly in $x\in \mathbb{R}^N$.\\

\begin{Remark}\label{H}
	It is worth to highlight the close relation between hypotheses $(V)$ and $(F_1)$. Since a linking geometry is sought, these assumptions provide the required features for the functional associated to problem (\ref{ES}). Thereby, note that, if instead of $F\geq 0$ as in $(F_1)$, it was considered $F\leq 0$, then in $(V)$ it would be asked $a_0>\lambda_1^2$. Indeed, in this case, problem (\ref{ES}) could be treated in the same way, provided that the positions of the equations were exchanged. 
\end{Remark}

\hspace{0.75cm}In view of $(F_1)$, given $\varepsilon>0$, for $2\leq p < +\infty$ if $N=2$ and for $2\leq p<2^\#$ if $N>2$, there exist constants $C>0$ and $\tilde{C}_{\varepsilon}>0$ such that
\begin{equation}\label{ese1}
\big|F(x,u(x),v(x))\big| \leq Ch(x)\Big(\varepsilon|(u(x),v(x))|^2 + \tilde{C}_\varepsilon|(u(x),v(x))|^p\Big),
\end{equation}
almost everywhere in $\mathbb{R}^N$. In addition, assuming also hypotheses $(V_1)_i-(V_3)$ for $i=1,2$ and $(V)$, the functional
$$\int_{\mathbb{R}^N}\big|F(x,u(x),v(x))\big|\ dx$$
belongs to $C^1(E,\mathbb{R})$, since for all $\varepsilon>0$ given it satisfies
\begin{equation}\label{ese2}
\int_{\mathbb{R}^N}\big|F(x,u(x),v(x))\big|\ dx\leq \varepsilon||(u,v)||^2_E + C_\varepsilon||(u,v)||^p_{E},
\end{equation}
for some constant $C_{\varepsilon} >0$ (cf. \cite{Ra}, Lemma 2.2 and also \cite{Sir}, Lemma 3.1).\\

\hspace{0.75cm}In virtue of $(V)$ and $(F_1)$ it is possible to assert that the growth of $F$ is controlled by the growth of $V_i$ by means of $h$, hence $F$ can be unbounded but under some restrictions. For instance, a function $F$ such that $$\Big(F_s(x,s,t),F_t(x,s,t)\Big) = c_0h(x)\left(\dfrac{s^3}{1+s^2},\dfrac{t^3}{1+t^2}\right),$$
or
$$\Big(F_s(x,s,t),F_t(x,s,t)\Big) = c_0h(x)\left(\dfrac{s(s^2+t^2)}{1+s^2+t^2},\dfrac{t(s^2+t^2)}{1+s^2+t^2}\right),$$
where $h(x)\geq a_0 >\lambda^1_1$, \  $V_i(x)\geq \big[h(x)\big]^{\alpha}$ for $|x|>r_0$, with $c_0, \ a_0, \ r_0$ positive constants and $\alpha >1$, satisfy assumptions $(V)$ and $(F_1)$.

\hspace{0.75cm}The main result of this section is stated below.

\begin{theorem}\label{est1}
	Under the assumptions $(V_1)_i-(V_3)_i, (V)$ and $(F_1)-(F_2)$, system (\ref{ES}) has a nontrivial solution in $H^1(\mathbb{R}^N,\mathbb{R}^2).$	
\end{theorem}

\subsection{Variational Framework}

\hspace{2.5cm}

\hspace{0.75cm} In order to tackle problem (\ref{ES}) by means of variational methods, the strongly indefinite functional $I:E\to \mathbb{R}$ given by
\small{
	\begin{eqnarray}\label{FI}
	I(u,v) &=& \dfrac{1}{2}\int_{\mathbb{R}^N}\Big(|\nabla u(x)|^2 + V_1(x)u^2(x)\Big)\,dx\nonumber\\
	&&- \dfrac{1}{2}\int_{\mathbb{R}^N}\Big(|\nabla v(x)|^2 +V_2(x)v^2(x)\
	\Big)\,dx - \int_{\mathbb{R}^N}F(x,u(x),v(x))\ dx\nonumber\\
	&=& \dfrac{1}{2} \Big(||(u,0)||^2_E - ||(0,v)||_E^2\Big) - \int_{\mathbb{R}^N}F(x,u(x),v(x))\ dx,
	\end{eqnarray}
}\normalsize
is associated to the elliptic system (\ref{ES}). Thereby, $I\in C^1(E,\mathbb{R})$ and

\begin{eqnarray}\label{DFI}
I'(u,v)(\varphi,\psi) &=& \int_{\mathbb{R}^N}\big(\nabla u(x) \cdot \nabla \varphi(x)+V_1(x)u(x)\varphi(x)\big)\,dx\nonumber\\ &&- \int_{\mathbb{R}^N}\big(\nabla v(x) \cdot \nabla \psi(x)+V_2(x)v(x)\psi(x)\big)\;dx\nonumber\\
&&- \int_{\mathbb{R}^N}F_u(x,u(x),v(x))\varphi(x)\ dx - \int_{ \mathbb{R}^N}F_v(x,u(x),v(x))\psi(x) \ dx\nonumber\\
&=& \Big((u,0),(\varphi,0)\Big)_E - \Big((0,v),(0,\psi)\Big)_E\nonumber\\
&&- \int_{\mathbb{R}^N}\Big(F_u(x,u(x),v(x))\varphi(x) + F_v(x,u(x),v(x))\psi(x)\Big)\;dx,
\end{eqnarray}
hence critical points of $I$ are weak solutions for (\ref{ES}). 

\hspace{0.75cm}To obtain a nontrivial critical point of the functional $I$ we make use of an abstract linking theorem proved by the authors in \cite{MS}, which we now recall.

\begin{theorem}\label{ALT}
	\textbf{Linking Theorem for Cerami Sequences:} Let $E$ be a real Hilbert space, with inner product $\big( \cdot, \cdot \big)$, $E_1$ a closed subspace of $E$ and $E_2=E_1^{\perp}$. Let $I \in C^1(E,\mathbb{R})$ satisfying:\\
	
	\hspace{-0.5cm}$(I_1) \ \ I(u)=\dfrac{1}{2}\big(Lu,u\big)+B(u),$ for all $u \in E$, where $u=u_1 +u_2 \in E_1 \oplus E_2$, $Lu=L_1u_1 + L_2u_2$ and\  $L_i: E_i \rightarrow E_i, \ i=1,2$ is a bounded linear self adjoint mapping.\\
	
	\hspace{-0.5cm}$(I_2) \ \ B$ is weakly continuous and uniformly differentiable on bounded subsets of $E$.\\
	
	\hspace{-0.5cm}$(I_3) \ $ There exist Hilbert manifolds $S,Q \subset E$, such that $Q$ is bounded and has boundary $\partial Q$, constants $\alpha > \omega$ and $v \in E_2$ such that\\
	$(i) \ S\subset v + E_1$ and $I \geq \alpha$ on $S$;\\
	$(ii) \ I\leq \omega$ on $\partial Q$;\\
	$(iii) \ S$ and $\partial Q$ link.\\
	
	\hspace{-0.5cm}$(I_4) \  \ $
	If for a sequence $(u_n)$, $I(u_n)$ is bounded and  $\left(1+||u_n||\right)||I'(u_n)|| \to 0$, as $n\to +\infty$, then $(u_n)$ is bounded.
	\vspace{0.5cm}\\
	\hspace{-0.5cm}Then I possesses a critical value $c\geq \alpha$.
\end{theorem}

\hspace{0.75cm}Thus, from now on, the plan is to find a critical point of $I$ by applying Theorem \ref{ALT}, since $I$ is an indefinite functional. Henceforth, it is necessary to check that $I$ satisfies $(I_1)-(I_4)$.

\hspace{0.75cm}First, it is convenient to analyse operator $A:= (A_1, -A_2)$ for $A_i$ defined in $(V_3)_i$ as an operator of $L^2(\mathbb{R}^N,\mathbb{R}^2)$ such that
$A(u,v):= (A_1u,-A_2v) \in L^2(\mathbb{R}^N,\mathbb{R}^2),$
for all $(u,v)$ in the domain of $A$. Moreover, for all $(u,v) \in E$, one has
\begin{equation}\label{ese3}
\Big(A(u,v),(u,v)\Big)_{L^2(\mathbb{R}^N,\mathbb{R}^2)}= \big|\big|(u,0)\big|\big|_E^2 - \big|\big|(0,v)\big|\big|_E^2.
\end{equation}
Hence, defining $$E_1 := \Big\{(u,0) \in H^1(\mathbb{R}^N,\mathbb{R}^2): \int_{\mathbb{R}^N}V_1(x)u^2(x)\ dx < +\infty\Big\}$$ and $$E_2:= \Big\{(0,v)\in H^1(\mathbb{R}^N,\mathbb{R}^2): \int_{\mathbb{R}^N}V_2(x)v^2(x)\ dx < +\infty\Big\},$$
it yields $E:= E_1\oplus E_2$. Thus, setting $L_1:=Id:E_1 \to E_1$, $L_2:= - Id:E_2 \to E_2$ and $L:E\to E$ with $L:=L_1+L_2$, where $Id$ is identity operator, it follows that
$$\Big(L(u,v),(\varphi,\psi)\Big)_E = \Big((u,0),(\varphi,0)\Big)_E - \Big((0,v),(0,\psi)\Big)_E.$$
Therefore, $I$ satisfies $(I_1)$ in Theorem \ref{ALT} since
\begin{eqnarray}\label{ese4}
\Big(L(u,v),(u,v)\Big)_E &=&\Big((u,0),(u,0)\Big)_E - \Big((0,v),(0,v)\Big)_E \nonumber \\
&=& ||(u,0)||_E^2-||(0,v)||_E^2 \nonumber\\
&=& \Big(A(u,v),(u,v)\Big)_{L^2(\mathbb{R}^N,\mathbb{R}^2)},
\end{eqnarray} 
and then
$$I(u,v)= \dfrac{1}{2}\Big(L(u,v),(u,v)\Big)_E +B(u,v),$$
where
\begin{equation}\label{B}
B(u,v):=-\displaystyle\int_{\mathbb{R}^N}F(x,u(x),v(x))\ dx.
\end{equation}

\begin{Remark}\label{esr1}
	Since each $A_i$ has discrete and positive spectrum $\sigma(A_i)=\{\lambda^i_n\}$ for $i=1,2$, $A$ \ has \ discrete \ spectrum $\sigma(A)= \displaystyle\{\lambda^1_n\}\displaystyle\cup\{-\lambda^2_n\}$. In fact, since $E$ is compactly embedded in $L^s_{h(x)}(\mathbb{R}^N,\mathbb{R}^2)$, it implies that $\sigma(A)$ is discrete (cf. \cite{Ra}, Proposition 2.4). In addition, \ $\lambda$ is an eigenvalue  of $A$ \ iff there exists \ $(u,v) \in E\setminus\{0\}$ such that $A(u,v)=\lambda(u,v)$, namely, $A_1u=\lambda u$ and $A_2v=(-\lambda)v$. Provided that, $u \not= 0$ or $v \not= 0$, then $\lambda$ is an eigenvalue of $A$ iff it is an eigenvalue of $A_1$ or $-\lambda$ is an eigenvalue of $A_2$. Therefore, $A$ has positive and negative spectra composed respectively by $\lambda^1_n \to +\infty$ and $-\lambda^2_n \to -\infty$ as $n \to +\infty.$ Furthermore, $\lambda^1_1$ is the smallest positive eigenvalue of $A$ and $-\lambda^2_1$ is the largest negative eigenvalue of $A$, and in virtue of $(V_3)_i$, $0\notin \sigma(A)$. In addition, $E_1$ and $E_2$ correspond to the subspaces of $E$ where $A$ is positive definite, and negative definite, respectively, and both are infinite dimensional.
\end{Remark}

\subsection{Establishing Compactness}

\hspace{2.5cm}

\hspace{0.75cm}With the aim of showing that $I$ satisfies $(I_2)$ in Theorem \ref{ALT}, the following lemmas are stated and proved.

\begin{lemma}\label{esl1}
	Under the hypotheses $(V_1)_i-(V_3)_i, (V)$ and $(F_1)$, functional $B:E\to \mathbb{R}$ defined as in (\ref{B})  is weakly continuous.	
\end{lemma}
\begin{proof}
	Let $(u_n,v_n) \rightharpoonup (u,v)$ weakly in $E$, since $E$ is compactly embedded in $L^s_{h(x)}(\mathbb{R}^N,\mathbb{R}^2)$ \ for $2\leq s<+\infty$ if $N=2$ and for $2\leq s < 2^\#$ if $N>2$, then $(u_n,v_n) \to (u,v)$, up to subsequences, \ strongly \ in \ $L^s_{h(x)}(\mathbb{R}^N,\mathbb{R}^2)$. In view of (\ref{ese1}), using standard arguments and applying Lebesgue Dominated Convergence Theorem, it yields
	$$\int_{\mathbb{R}^N}F(x,u_n(x),v_n(x))\ dx \to \int_{\mathbb{R}^N}F(x,u(x),v(x))\ dx,$$
	as $n \to +\infty,$ namely, $B(u_n,v_n)\to B(u,v)$ and $B$ is weakly continuous.
\end{proof}

\begin{lemma}\label{esl2}
	Assuming that $(V_1)_i-(V_3)_i, (V)$ and $(F_1)$ hold, functional $I$ is uniformly differentiable on bounded sets of $E$.
\end{lemma}
\begin{proof}
	Fixed $R>0$ and given $(u_1,u_2)+(v_1,v_2), \ (v_1,v_2) \in B_R \subset E$, the closed ball centered on the origin, one has
	\small{
		\begin{eqnarray}\label{esl2e1}
		&&\Big|B(u_1 +v_1,u_2+v_2)-B(u_1,u_2)-B'(u_1,u_2)(v_1,v_2)\Big|\nonumber\\
		&=& \int_{\mathbb{R}^N}\Big|F\big(x,u_1 +v_1,u_2 +v_2\big)- F\big(x,u_1,u_2\big)- F_u\big(x,u_1,u_2\big)v_1- F_v\big(x,u_1,u_2\big)v_2\Big|dx  \nonumber\\
		&\leq& \int_{\mathbb{R}^N}\dfrac{1}{h(x)^{\frac{1}{2}}}\Big|F_u\big(x,z_1,z_2\big) + F_v\big(x,z_1,z_2\big) - F_u\big(x,u_1,u_2\big)- F_v\big(x,u_1,u_2\big)\Big|\ \big|h(x)^{\frac{1}{2}}(v_1,v_2)\big|dx \nonumber\\
		&\leq& C_2||\xi||_{L^2(\mathbb{R}^N)}||(v_1,v_2)||_E, \nonumber\\
		\end{eqnarray}
	}\normalsize
	where
	\small{
		$$\xi(x):= \dfrac{1}{h(x)^{\frac{1}{2}}}\Big|F_u\big(x,z_1(x),z_2(x)\big) + F_v\big(x,z_1(x),z_2(x)\big) - F_u\big(x,u_1(x),u_2(x)\big)- F_v\big(x,u_1(x),u_2(x)\big)\Big|$$}
	\normalsize
	\hspace{-0.1cm}and $(z_1(x),z_2(x))= (u_1(x),u_2(x))+(\theta_1(x)v_1(x),\theta_2(x)v_2(x))$, where $0\leq\theta_i(x)\leq 1$ is given by Mean Value Theorem and $C_2>0$ is the constant given by the \ continuous embedding $ E \hookrightarrow L^2_{h(x)}(\mathbb{R}^N,\mathbb{R}^2)$.
	So as to prove \ that $B$ is uniformly \ differentiable \ on bounded sets of $E$, \ it is sufficient to show that given $\varepsilon>0$ there exist $\delta>0$ such that $C_2||\xi||_{L^2(\mathbb{R}^N)}\leq \varepsilon$ for all $(u_1,u_2)+(v_1,v_2), \ (v_1,v_2) \in B_R \subset E$ with $||(v_1,v_2)||_E\leq \delta$. However, by definition of limit, this is equivalent to show that $\xi\to 0$ in $L^2(\mathbb{R}^N)$ as $(v_1,v_2) \to 0$ in $E$. 
	
	\hspace{0.75cm}Observe \ that \ $(v_1,v_2) \to 0$ in $E$ \ implies \ $(v_1,v_2)\to 0$ in $L^s_{h(x)}(\mathbb{R}^N,\mathbb{R}^2)$ \ for $2\leq s <+\infty$ if $N=2$ and for $2\leq s < 2^\#$ if $N>2$, \ then \ $(v_1(x),v_2(x))\to 0$ almost everywhere in $\mathbb{R}^N$ and \ there \ exist \ $(\varphi_s,\psi_s) \in L^s_{h(x)}(\mathbb{R}^N,\mathbb{R}^2)$ \ such \ that $|v_1(x)|\leq \varphi_s(x)$, $|v_2(x)|\leq \psi_s(x)$ almost everywhere in $\mathbb{R}^N$. Moreover, $(z_1(x),z_2(x))\to (u_1(x),u_2(x))$ almost everywhere in $\mathbb{R}^N$. Thus, $\xi(x)\to 0$ almost everywhere in $\mathbb{R}^N$. Furthermore, from $(F_1)$ and L'Hospital rule, one has
	\small{
		$$\big|F_u(x,z_1(x),z_2(x))\big|\leq Ch(x)\big|(z_1(x),z_2(x))\big| \quad \text{and} \quad \big|F_v(x,z_1(x),z_2(x))\big|\leq Ch(x)\big|(z_1(x),z_2(x))\big|,$$}
	\normalsize
	\hspace{-0.1cm}almost everywhere in $\mathbb{R}^N$, for some constant $C>0$ and the analogous holds for $u_i$ instead of $z_i$. Then, for $s=2$ it follows that
	\begin{eqnarray}\label{esl2e2}
	|\xi(x)|^2 &\leq& \dfrac{2}{h(x)}\left[\Big|F_u\big(x,z_1(x),z_2(x)\big)\Big| + \Big|F_v\big(x,z_1(x),z_2(x)\big)\Big|\right]^2 \nonumber\\
	&&+ \, \dfrac{2}{h(x)}\left[ \Big|F_u\big(x,u_1(x),u_2(x)\big)\Big| + \Big|F_v\big(x,u_1(x),u_2(x)\big)\Big| \right]^2\nonumber\\
	&\leq& \dfrac{2}{h(x)}\left[\Big(2Ch(x)|(z_1(x),z_2(x))|\Big)^2+\Big(2Ch(x)|(u_1(x),u_2(x))|\Big)^2\right]\nonumber\\
	&\leq&24C^2\Big[h(x)|(\varphi_2(x),\psi_2(x))|^2+h(x)|(u_1(x),u_2(x))|^2\Big].
	\end{eqnarray}
	Therefore, by applying Lebesgue Dominated Convergence Theorem it yields
	$$||\xi||^2_{L^2(\mathbb{R}^N)}=\int_{\mathbb{R}^N}|\xi(x)|^2\ dx \to 0,$$
	as $(v_1,v_2)\to 0$ in $E$ and the result hold.
\end{proof}

\subsection{Linking Geometry}

\hspace{2.5cm}

\hspace{0.75cm}Under the purpose of proving that $I$ satisfies $(I_3)$ in Theorem \ref{ALT}, as usual, set
$$S:= (\partial B_{\rho}\cap E_1) \quad \text{and} \quad
{Q:= \{(re,v): r\geq0, (0,v) \in E_2, ||(re,v)||_E\leq r_1\}},$$
where $0<\rho < r_1$ are constants, $(e,0)\in E_1$ and $e$ is an eigenfunction associated to $\lambda_1^1$ the first eigenvalue of $A_1$, with $||(e,0)||_E=1$, therefore, such $S$ and $Q$ ``link''. Due to the strict inequality $a_0>\lambda_1^1$ in hypothesis $(V)$ and from (\ref{ese3}), such an $e$ satisfies
\begin{eqnarray} \label{ese5}
1&=&||(e,0)||_E^2= \Big(A(e,0),(e,0)\Big)_{L^2(\mathbb{R}^N,\mathbb{R}^2)} \nonumber\\
&=& \int_{\mathbb{R}^N}\big(|\nabla e(x)|^2 +V_1(x)e^2(x)\big)\,dx = (A_1e,e)_{L^2(\mathbb{R}^N)} \nonumber \\
&=& \lambda_1^1||e||^2_{L^2(\mathbb{R}^N)}\nonumber\\
&<& a_0||(e,0)||^2_{L^2(\mathbb{R}^N,\mathbb{R}^2)}\nonumber \\
&\leq&\int_{\mathbb{R}^N}h(x)|(e(x),0)|^2 \ dx.
\end{eqnarray}
Thus, by means of (\ref{ese5}) it is possible to show that if $r_1>0$ is large enough, then $I|_S\geq \alpha >0$ and $I|_{\partial Q}\leq 0$ for some $\alpha>0$. Thus, $I$ satisfies $(I_3)$ for some $\alpha>0, \ \omega=0$ and arbitrary $(0,v) \in E_2$. Indeed, next lemma gives the result.

\begin{lemma} \label{esl3}
	Suppose that $(V_1)_i-(V_3)_i$, $(V)$ and $(F_1)$ hold, then $I$ satisfies $(I_3)$.	
\end{lemma}	
\begin{proof}
	Note that $S\subset E_1$, then from (\ref{FI}) and (\ref{ese2}), for all $(u_1,0) \in S$, it yields
	\begin{eqnarray}\label{esl3e1}
	I(u_1,0) &\geq& \dfrac{1}{2}\rho^2 - \Big(\varepsilon||(u_1,0)||^2_E + C_\varepsilon||(u_1,0)||^p_{E}\Big)\nonumber\\
	&=& \rho^2\left[\left(\dfrac{1}{2}-\varepsilon \right) - C_{\varepsilon}\rho^{p-2}\right]\nonumber \\
	&\geq&
	\alpha>0,
	\end{eqnarray}
	where $\varepsilon, \rho$ are chosen small enough.
	Therefore, from (\ref{esl3e1}), $(I_3) \ (i)$ holds for $I$.
	
	\hspace{0,75cm}With the purpose of proving that \ $I$ satisfies \ $(I_3) \ (ii)$ in \ Theorem \ref{ALT},\  with $\omega = 0$, \ observe that \ $I(0,v)\leq 0$, \ for all $(0,v) \in E_2$, then it suffices to show that $I(re,v)\leq 0$ for $r>0, \ (0,v) \in E_2$ and $||(re,v)||_E\geq r_1$, for some $r_1>0$ \ large enough. \ Arguing indirectly assume that some sequence $(r_ne,v_n)\subset \mathbb{R}^+e \oplus E_2$ satisfies $ ||(r_ne,v_n)||_E\to +\infty$ and $I(r_ne,v_n)>0$ for all $n \in \mathbb{N}$. Seeking a contradiction, set $$ (s_ne,w_n):= \dfrac{(r_ne,v_n)}{||(r_ne,v_n)||_E},$$
	where $ s_n \in \mathbb{R}^+, (0,w_n) \in E_2$ and $||(s_ne,w_n)||_E=1$. Provided that $(s_ne,w_n)$ is bounded, up to subsequences, it follows that $(s_ne,w_n) \rightharpoonup (se,w)$ in $E$, hence $(s_ne,w_n)\to (se,w)$ in $L^2_{h(x)}(\mathbb{R}^N,\mathbb{R}^2)$, due to the compact embedding $E \hookrightarrow L^2_{h(x)}(\mathbb{R}^N,\mathbb{R}^2)$. Then, up to subsequences, $(s_ne(x),w_n(x)) \to (se(x),w(x))$ almost everywhere in $\mathbb{R}^N$, $s_n \to s$ in $\mathbb{R}^+$ and $(0,w_n) \rightharpoonup (0,w)$ in $E_2$, since $s_n$ and $(0,w_n)$ are also bounded. Noting that
	$$1 = ||(s_ne,w_n)||_E^2= s_n^2 +||(0,w_n)||_E^2,$$
	it follows that $0\leq s^2_n\leq 1$, and it yields
	\begin{eqnarray}\label{esl3e2}
	\dfrac{I(r_ne,v_n)}{||(r_ne,v_n)||_E^2} &=& \dfrac{1}{2}\Big({s^2_n}||(e,0)||_E^2 - ||(0,w_n)||_E^2\Big) - \int_{\mathbb{R}^N}\dfrac{F(x,r_ne(x),v_n(x))}{||(r_ne,v_n)||_E^2}\ dx\nonumber\\
	&=& s^2_n - \dfrac{1}{2} - \int_{\mathbb{R}^N}\dfrac{F(x, r_ne(x),v_n(x))}{||(r_ne,v_n)||_E^2}\ dx  >0,
	\end{eqnarray}
	hence $0<s\leq1$. 
	Moreover, from (\ref{ese5}) one has $1< \displaystyle\int_{\mathbb{R}^N}h(x)|(e(x),0)|^2 \ dx.$ Then, 
	\begin{eqnarray}\label{esl3e3}
	0&>&s^2-s^2\displaystyle\int_{\mathbb{R}^N}h(x)|(e(x),0)|^2 \ dx\nonumber\\
	&>& s^2\left(1 - \displaystyle\int_{\mathbb{R}^N}h(x)|(e(x),0)|^2 \ dx\right) - \dfrac{1}{2} - \displaystyle\int_{\mathbb{R}^N}h(x)|(0,w(x))|^2 \ dx.
	\end{eqnarray}
	On the other hand, since $(s_ne,v_n)$ is convergent in $L^2_{h(x)}(\mathbb{R}^N,\mathbb{R}^2)$, there exists some $(\varphi_2,\psi_2) \in L^2_{h(x)}(\mathbb{R}^N,\mathbb{R}^2)$ such that $|s_ne(x)|\leq \varphi_2(x)$ and $|w_n(x)|\leq \psi_2(x)$ almost everywhere in $\mathbb{R}^N$ and from (\ref{ese1}) with $p=2$, it follows that
	$$\left|\dfrac{F(x ,r_ne(x),v_n(x))}{||(r_ne,v_n)||_E^2}\right|\leq \tilde{C}h(x)|(s_ne(x),w_n(x))|^2\leq \tilde{C}h(x)|(\varphi_2(x),\psi_2(x))|^2,$$
	almost everywhere in $\mathbb{R}^N$, for some constant $\tilde{C}>0$. Provided that $||(r_ne, v_n)||_E\to +\infty$, and $(s_ne(x),w_n(x))\to (se(x),w(x))\not= 0$, almost everywhere in $supp(e)$,  it follows that $(r_ne(x),v_n(x))=(s_ne(x),w_n(x))||(r_ne,v_n)||_E\to +\infty$ almost everywhere in $supp(e)$, as $n\to +\infty$, hence in view of $(F_1)$ one has
	\begin{eqnarray*}
	\dfrac{F(x,r_ne(x),v_n(x))}{||(r_ne,v_n)||_E^2} &=& \dfrac{F\big(x,||r_ne+u_n||_E(s_ne(x),w_n(x)\big)|(s_ne(x),w_n(x))|^2}{||r_ne+u_n||_E^2|(s_ne(x),w_n(x))|^2}\\
	&=& h(x)|(se(x),w(x))|^2 + o_n(1),
	\end{eqnarray*}
	almost everywhere in $supp(se,w)$ as $n \to +\infty$. Note that, $supp(se,w) \not=\emptyset$, because $supp(e)\not= \emptyset$. Thus, by Lebesgue Dominated Convergence Theorem,
	
	$$\int_{\mathbb{R}^N}\dfrac{F(x,r_ne(x),v_n(x))}{||(r_ne,v_n)||_E^2}dx \to \int_{\mathbb{R}^N}h(x)\big|(se(x), w(x))\big|^2 \ dx,$$\\
	as $n \to +\infty.$ From (\ref{esl3e2}) one has
	$$s^2_n - \dfrac{1}{2} - \int_{\mathbb{R}^N}\dfrac{F(x, r_ne(x),v_n(x))}{||(r_ne,v_n)||_E^2}\ dx >0.$$	
	Passing to the limit as $n \to +\infty$, it yields
	\begin{eqnarray}\label{esl3e4}
	0&\leq& s^2 - \dfrac{1}{2} - \int_{\mathbb{R}^N}h(x)\big|(se(x), w(x))\big|^2 \ dx \nonumber\\ &=& s^2\left(1 - \displaystyle\int_{\Omega}h(x)|(e(x),0)|^2 \ dx\right) - \dfrac{1}{2} - \displaystyle\int_{\Omega}h(x)|(0,w(x))|^2 \ dx,\\
	\nonumber
	\end{eqnarray}
	which is contrary to (\ref{esl3e3}). Therefore the result holds.
\end{proof}

\subsection{Boundedness of Cerami Sequences}

\hspace{2.5cm}

\hspace{0.75cm}Next lemma ensures $I$ satisfies last hypothesis in Theorem \ref{ALT}. Finally, with this result it will be possible to prove Theorem \ref{est1}.

\begin{lemma}	\label{esl4}
	Assuming that $(V_1)_i-(V_3)_i$, $(V)$ and $(F_1)-(F_2)$ hold, $I$ satisfies $(I_4)$.	
\end{lemma}

\begin{proof}
	Let $b>0$ be \ an \ arbitrary \ constant, \ and \ take \ $(u_n,v_n) \subset I^{-1}([c-b,c+b])$ \ such that $\big(1+||(u_n,v_n)||_E\big)||I'(u_n,v_n)||_{E'}\to0$, it is necessary to show that $(u_n,v_n)$ is bounded. Suppose by contradiction that $||(u_n,v_n)||_E\to +\infty$, up to subsequences. Setting $(\tilde{u}_n,\tilde{v}_n):=\dfrac{(u_n,v_n)}{||(u_n,v_n)||_E}$, it is bounded, hence $(\tilde{u}_n,\tilde{v}_n) \rightharpoonup (\tilde{u},\tilde{v})$ in $E$ and due to the compact embeddings $E\hookrightarrow L^s_{h(x)}(\mathbb{R}^N,\mathbb{R}^2)$ for $2\leq s<+\infty$ if $N=2$ and $2\leq s < 2^\#$ if $N>2$, it implies that $(\tilde{u}_n,\tilde{v}_n) \to (\tilde{u},\tilde{v})$ in $L^s_{h(x)}(\mathbb{R}^N,\mathbb{R}^2)$. Writing $(u_n,v_n)= (u_n,0)+(0,v_n) \in E_1\oplus E_2$, it satisfies
	\begin{eqnarray}\label{esl4e1}
	o_n(1)&=& \dfrac{1}{||(u_n,v_n)||_E}I'(u_n,v_n)(\tilde{u}_{n},0)\nonumber\\
	&=&||(\tilde{u}_{n},0)||_E^2 - \int_{\mathbb{R}^N}\dfrac{F_u(x,u_n(x),v_n(x))}{||(u_n,v_n)||_E}\tilde{u}_n(x) \ dx \nonumber\\
	&=&||(\tilde{u}_{n},0)||_E^2 - \int_{\mathbb{R}^N}\dfrac{F_u(x,u_n(x),v_n(x))}{u_n(x)}\tilde{u}^2_n(x) \ dx
	\end{eqnarray}
	and
	\begin{eqnarray}\label{esl4e2}
	o_n(1)&=& \dfrac{1}{||(u_n,v_n)||_E}I'(u_n,v_n)(0,\tilde{v}_{n})\nonumber\\
	&=&-||(0,\tilde{v}_{n})||_E^2 - \int_{\mathbb{R}^N}\dfrac{F_v(x,u_n(x),v_n(x))}{v_n(x)}\tilde{v}^2_n(x) \ dx. \\
	\nonumber
	\end{eqnarray}
	Subtracting (\ref{esl4e2}) from (\ref{esl4e1}), and using that $1 = ||(\tilde{u}_n,0)||_E^2 + ||(0,\tilde{v}_n)||_E^2$, it yields\\
	\begin{equation}\label{esl4e3}
	o_n(1)=1 - \int_{\mathbb{R}^N}\left[\dfrac{F_u(x,u_n(x),v_n(x))}{u_n(x)}\tilde{u}^2_n(x) - \dfrac{F_v(x,u_n(x),v_n(x))}{v_n(x)}\tilde{v}^2_n(x) \right]dx.\\
	\end{equation}\\
	Provided that $(\tilde{u}_n,\tilde{v}_n)\to(\tilde{u},\tilde{v})$ in $L_{h(x)}^2(\mathbb{R}^N,\mathbb{R}^2)$, there \ exist \ $(\varphi_2,\psi_2) \in L^2_{h(x)}(\mathbb{R}^N,\mathbb{R}^2)$ \ such \ that $|\tilde{u}_n(x)|\leq \varphi_2(x)$  and $|\tilde{v}_n(x)|\leq \psi_2(x)$, almost everywhere in $\mathbb{R}^N$. Furthermore, from $(F_1)$ it follows that
	
	$$\left|\dfrac{F_u(x,u_n(x),v_n(x))}{u_n(x)}\right|\leq Ch(x)\quad \text{and}\quad \left|\dfrac{F_v(x,u_n(x),v_n(x))}{v_n(x)}\right|\leq Ch(x),$$
	for some constant $C>0$, almost everywhere in $\mathbb{R}^N$. Hence,
	\small{
		\begin{eqnarray*}
			\left|\dfrac{F_u(x,u_n(x),v_n(x))}{u_n(x)}\tilde{u}^2_n(x)- \dfrac{F_v(x,u_n(x),v_n(x))}{v_n(x)}\tilde{v}^2_n(x) \right|
			&\leq& Ch(x)\Big(\tilde{u}^2_n(x) + \tilde{v}^2_n(x)\Big) \nonumber\\
			&\leq&Ch(x)\Big(\varphi^2_2(x) + \psi^2_2(x)\Big).
		\end{eqnarray*}
	}\normalsize
	Since $(\tilde{u}_n,\tilde{v}_n) \to (\tilde{u},\tilde{v})$ in $L^2_{h(x)}(\mathbb{R}^N,\mathbb{R}^2)$, $|(u_n(x),v_n(x))|\to +\infty$, for all $x \in \mathbb{R}^N$ such that $(\tilde{u}(x),\tilde{v}(x))\not=0$, then from $(F_1)$ and L'Hospital rule, it follows that
	\begin{equation*}
	\dfrac{F_u(x,u_n(x),v_n(x))}{u_n(x)}\tilde{u}^2_n(x)- \dfrac{F_v(x,u_n(x),v_n(x))}{v_n(x)}\tilde{v}^2_n(x) \to 2h(x)\left[\tilde{u}^2(x)-\tilde{v}^2(x)\right],
	\end{equation*}
	as $n \to +\infty$, for all $x \in \mathbb{R}^N$ $(\tilde{u}(x),\tilde{v}(x))\not=0$. Therefore, by Lebesgue Dominated Convergence Theorem one has
	\small{
		\begin{equation} \label{esl4e4}
		\int_{\mathbb{R}^N}\left[\dfrac{F_u(x,u_n(x),v_n(x))}{u_n(x)}\tilde{u}^2_n(x) - \dfrac{F_v(x,u_n(x),v_n(x))}{v_n(x)}\tilde{v}^2_n(x) \right]dx \to \int_{\mathbb{R}^N}2h(x)\left[\tilde{u}^2(x)-\tilde{v}^2(x)\right]dx.
		\end{equation}
	}\normalsize
	Hence, passing to the limit in (\ref{esl4e3}) as $n \to +\infty$, it yields
	\begin{equation}\label{esl4e5}
	\int_{\mathbb{R}^N}2h(x)\left[\tilde{u}^2(x)-\tilde{v}^2(x)\right]dx = 1,
	\end{equation}
	which implies there exists $\Omega \subset \mathbb{R}^N$ such that $|\Omega|>0$ and $(\tilde{u}(x),\tilde{v}(x))\not=0$ for all $x \in \Omega$, otherwise (\ref{esl4e5}) yields a contradiction. 
	
	\hspace{0.75cm}On the other hand, since $(u_n,v_n)$ is a Cerami sequence, there exists a constant $M_0>0$ such that
	\begin{eqnarray}\label{esl4e6}
	M_0 &\geq& I(u_n,v_n) - \dfrac{1}{2}I'(u_n,v_n)(u_n,v_n)\nonumber\\
	&=& \dfrac{1}{2} \int_{\mathbb{R}^N} \Big(F_u(x,u_n,v_n)u_n - F_v(x,u_n,v_n)v_n - 2F(x,u_n,v_n)\Big)\,dx\nonumber\\
	&\geq&\dfrac{1}{2} \int_{\Omega} Q(x,u_n(x),v_n(x))\ dx.
	\end{eqnarray}
	Provided that  $|(u_n(x),v_n(x))|\to +\infty$, for all $x \in \Omega$, in view of $(F_2)$, it yields
	$$\dfrac{1}{2} \int_{\Omega} Q(x,u_n(x),v_n(x))\ dx \to +\infty,$$
	as $n \to +\infty.$ Passing (\ref{esl4e6}) to the limit, a contradiction holds.
\end{proof}

\hspace{0.75cm}At long last, the main result is proved.

\begin{proof}[Proof of Theorem \ref{est1}]
	Provided that $I$ satisfies all assumptions $(I_1)-(I_4)$ in Theorem \ref{ALT}, it ensures a critical point  $(u,v) \in E$ of $I$, with $I(u,v)=c\geq\alpha>0$. It implies that $I'(u,v)(\varphi,\psi) = 0$, for all $(\varphi,\psi) \in E$ and $I(u,v) >=I(0,0)$, thus $(u,v)\not=0$. Since $I\in C^1(E,\mathbb{R})$, it yields that $(u,v)$ is a nontrivial solution for system (\ref{ES}).
\end{proof}

\section{Super Quadratic Elliptic Systems}

\hspace{0.75cm}The following assumptions on the Hamiltonian $F$ are required for the super quadratic case.\\

\hspace{-0.4cm}$(F'_1)$ \ $F(x,s,t) \in C^1(\mathbb{R}^N\times\mathbb{R}^2,\mathbb{R}^+)$,
$$\dfrac{\left|\big(F_s(x,s,t),F_t(x,s,t)\big)\right|}{h(x)} = o(|(s,t)|) \quad \text{as} \quad |(s,t)|\to 0, \quad \text{uniformly \ in} \quad x,$$
and for some $2< p <+\infty$ if $N=2$ or $2< p <2^\#:= \dfrac{4}{\alpha(N-2)}$ if $N>2$,
$$\left|\big(F_s(x,s,t),F_t(x,s,t)\big)\right|\leq c_0h(x)\big(1+|(s,t)|^{p-1}\big) , \quad \text{for \ all} \quad (x,s,t)\in \mathbb{R}^N \times\mathbb{R}^2;$$

\hspace{-0.8cm} \ \ $(F'_2)$ $\displaystyle\lim_{|(s,t)| \to +\infty}\dfrac{F(x,s,t)}{|(s,t)|^2}= +\infty, \quad \text{uniformly \ in}\quad x.$\\

\hspace{0.75cm}Note that for such a $p$ given in $(F'_1)$ and arbitrary $\varepsilon>0$, there exist constants $C>0, \ \tilde{C}_{\varepsilon}>0$ such that (\ref{ese1}) and (\ref{ese2}) hold. Hence,
the functional
$$\int_{\mathbb{R}^N}\big|F(x,u(x),v(x))\big|\ dx$$
belongs to $C^1(E,\mathbb{R})$
(cf. \cite{Ra}, Lemma 2.2 and also \cite{Sir}, Lemma 3.1).

\hspace{0.75cm}In addition, for $Q(x,u,v)$ defined as in $(F_2)$, so as to ensure the boundedness of Cerami sequences, the following condition is required.\\

\hspace{-0.7cm} \ $(F'_3)$  There exists a constant $D\geq1$ such that 
$$Q(x,w,z)\leq DQ(x,u,v) ,$$
for all $(u,v), (w,z) \in \mathbb{R}^2$ with $ |(w,z)|\leq |(u,v)|$.
\vspace{0.5cm}

\hspace{0.75cm}It is worth to point out that hypothesis $(F'_3)$ was introduced by Jeanjean and Tanaka \cite{JeTa}, in order to bound Cerami sequences in cases where the so called Ambrosetti-Rabinowitz condition \cite{AR} is not satisfied by the nonlinearity, for example, when	$F(x,s,t) = c_0h(x)(\log|(s,t)|)^{\frac{1}{2}}(s^2,t^2)$, for $|(s,t)|\geq1$. An alternative condition for the boundedness is mentioned at the end of this work.

\hspace{0.75cm}Similarly to the asymptotically quadratic case, the relation between $(V)$ and $(F'_1)$ shows that the growth of $F$ is controlled by the growth of $V_i$ by means of $h$, hence $F$ can be unbounded but under some restrictions. For instance, a function $F$ such that $$\Big(F_s(x,s,t),F_t(x,s,t)\Big) = c_0h(x)\left(|(s,t)|^{p-2}s,|(s,t)|^{p-2}t\right),$$
where $h(x)\geq a_0 >\lambda^1_1$, \  $V_i(x)\geq \big[h(x)\big]^{\alpha}$ for $|x|>r_0$, with $c_0, \ a_0, \ r_0$ positive constants and $\alpha >1$, satisfy assumptions $(V)$ and $(F_1)$.

\hspace{0.75cm}Now, the main result of this section is stated.

\begin{theorem}\label{est2}
	Suppose that $(V_1)_i-(V_3)_i, (V)$ and $(F'_1)-(F'_3)$ hold, then system (\ref{ES}) has a nontrivial solution in $H^1(\mathbb{R}^N,\mathbb{R}^2).$	
\end{theorem}

\hspace{0.75cm}Under the purpose of proving Theorem \ref{est2}, the variational structure established in Subsection 2.1. is assumed here. Thereby, $I$ satisfies $(I_1)$ in Theorem \ref{ALT}, which must be applied so as to get a non trivial critical point of $I$. Furthermore, the conclusions of Lemmas \ref{esl1} and \ref{esl2}, in Subsection 2.2, are exactly the same replacing $(F_1)$ by $(F'_1)$. Thus, $(I_2)$ is satisfied by $I$ as well.

\subsection{Linking Structure}

\hspace{2.5cm}

\hspace{0.75cm}In order to prove that $I$ satisfies $(I_3)$ in Theorem \ref{ALT}, all arguments in Subsection 2.3 are adapted to the super quadratic case. As before,
$$S:= (\partial B_{\rho}\cap E_1) \quad \text{and} \quad
{Q:= \{(re,v): r\geq0, (0,v) \in E_2, ||(re,v)||_E\leq r_1\}},$$
where $0<\rho < r_1$ are constants and $(e,0)\in E_1$ is an arbitrary vector
with $||(e,0)||_E=1$, therefore, such $S$ and $Q$ ``link''.
Next lemma shows that if $r_1>0$ is large enough, then $I|_S\geq \alpha >0$ and $I|_{\partial Q}\leq 0$ for some $\alpha>0$. Thus, $I$ satisfies $(I_3)$ for some $\alpha>0, \ \omega=0$ and arbitrary $(0,v) \in E_2$.

\begin{lemma} \label{esl32}
	Assume that $(V_1)_i-(V_3)_i$, $(V)$ and $(F'_1)-(F'_2)$ hold, then $I$ satisfies $(I_3)$.	
\end{lemma}	

\hspace{0.75cm}Although $(I_3) \ (i)$ holds for $I$ by the same argument found in Lemma \ref{esl3}, for the sake of completeness, this proof is repeated here.
\begin{proof}
	From (\ref{FI}) and (\ref{ese2}), \ for all $(u_1,0) \in S$, it yields (\ref{esl3e1}) and likewise $(I_3) \ (i)$ holds for $I$.
	
	\hspace{0,75cm}For showing \ $(I_3) \ (ii)$  with $\omega = 0$, since $I(0,v)\leq 0$, \ for all $(0,v) \in E_2$, it is enough to prove that $I(re,v)\leq 0$ for $r>0, \ (0,v) \in E_2$ and $||(re,v)||_E\geq r_1$, for some $r_1>0$ \  sufficiently \ large. Seeking a contradiction, assume that some sequence $(r_ne,v_n)\subset \mathbb{R}^+e \oplus E_2$ satisfies $ ||(r_ne,v_n)||_E\to +\infty$ and $I(r_ne,v_n)>0$ for all $n \in \mathbb{N}$. Set $$ (s_ne,w_n):= \dfrac{(r_ne,v_n)}{||(r_ne,v_n)||_E},$$
	where $ s_n \in \mathbb{R}^+, (0,w_n) \in E_2$ and $||(s_ne,w_n)||_E=1$. Provided that $(s_ne,w_n)$ is bounded, up to subsequences, it follows that $(s_ne,w_n) \rightharpoonup (se,w)$ in $E$, hence $(s_ne,w_n)\to (se,w)$ in $L^2_{h(x)}(\mathbb{R}^N,\mathbb{R}^2)$, due to the compact embedding $E \hookrightarrow L^2_{h(x)}(\mathbb{R}^N,\mathbb{R}^2)$. Then, up to subsequences, $(s_ne(x),w_n(x)) \to (se(x),w(x))$ almost everywhere in $\mathbb{R}^N$, $s_n \to s$ in $\mathbb{R}^+$ and $(0,w_n) \rightharpoonup (0,w)$ in $E_2$. Since 
	$1 = ||(s_ne,w_n)||_E^2= s_n^2 +||(0,w_n)||_E^2,$
	it follows that $0\leq s^2_n\leq 1$, and it yields
	\begin{eqnarray}\label{esl32e2}
	0<\dfrac{I(r_ne,v_n)}{||(r_ne,v_n)||_E^2} &=& \dfrac{1}{2}\Big({s^2_n}||(e,0)||_E^2 - ||(0,w_n)||_E^2\Big) - \int_{\mathbb{R}^N}\dfrac{F(x,r_ne(x),v_n(x))}{||(r_ne,v_n)||_E^2}\ dx\nonumber\\
	&=& s^2_n - \dfrac{1}{2} - \int_{\mathbb{R}^N}\dfrac{F(x, r_ne(x),v_n(x))}{||(r_ne,v_n)||_E^2}\ dx\nonumber\\
	&=& s^2_n - \dfrac{1}{2} - \int_{\mathbb{R}^N}\dfrac{F(x, r_ne(x),v_n(x))}{|(r_ne(x),v_n(x))|^2}{|(s_ne(x),w_n(x))|^2}\ dx,
	\end{eqnarray}
	hence $0<s\leq1$. 
	Provided that $||(r_ne, v_n)||_E\to +\infty$ and $$(s_ne(x),w_n(x))\to (se(x),w(x))\not= 0$$
	almost everywhere in supp$(e)$,  it follows that
	
	$$(r_ne(x),v_n(x))=(s_ne(x),w_n(x))||(r_ne,v_n)||_E\to +\infty,$$
	almost everywhere in supp$(e)$, as $n\to +\infty$, hence in view of $(F'_2)$ one has
	$$\dfrac{F(x, r_ne(x),v_n(x))}{|(r_ne(x),v_n(x))|^2}{|(s_ne(x),w_n(x))|^2} \to +\infty,$$
	\hspace{-0.1cm}almost everywhere in supp$(se,w)$ as $n \to +\infty$. Note that, $supp(se,w) \not=\emptyset$, because supp$(e)\not= \emptyset$. Thus, by Fatou's Lemma,
	\begin{equation}\label{esl32e4}
	\int_{\mathbb{R}^N}\dfrac{F(x, r_ne(x),v_n(x))}{|(r_ne(x),v_n(x))|^2}{|(s_ne(x),w_n(x))|^2}\ dx\to +\infty,
	\end{equation}
	as $n \to +\infty.$ In virtue of (\ref{esl32e2}) one has
	$$\dfrac{1}{2}\geq s^2_n - \dfrac{1}{2} > \int_{\mathbb{R}^N}\dfrac{F(x, r_ne(x),v_n(x))}{|(r_ne(x),v_n(x))|^2}{|(s_ne(x),w_n(x))|^2}\ dx.$$	
	Due to (\ref{esl32e4}), passing to the limit as $n \to +\infty$ it yields a contradiction.
\end{proof}

\begin{Remark}\label{esr3}
	Note that the assumption $a_0>\lambda_1^1$ in $(V)$, is irrelevant in the super quadratic case, since inequality (\ref{ese5}) is not used to establish the linking geometry as was done in Lemma \ref{esl3}. Therefore, for the super quadratic case, it is just necessary to require $a_0>0$ in $(V)$. Thereby, if  $F\leq0$ were considered instead of $F\geq0$, to treat the problem it would be necessary only to exchange one equation for the another, requiring hypothesis $(F'_2)$ for $-F\geq0$ instead of $F$.
\end{Remark}

\subsection{Boundedness of Cerami Sequences}

\hspace{2.5cm}

\hspace{0.75cm}Last hypothesis in Theorem \ref{ALT} is ensured in this subsection, and subsequently Theorem \ref{est2} is going to be proved.

\begin{lemma}	\label{esl5}
	Supposing that $(V_1)_i-(V_3)_i, (V),  (F'_1)-(F'_2)$ and $(F''_3)$ hold true for $I$, it satisfies $(I_4)$.
\end{lemma}
\begin{proof}
	Let \ $b>0$ \ be \ a \ constant \ such \ that \ $c-b>0$ \ and \ take \ ${(u_n,v_n) \subset I^{-1}([c-b,c+b])}$ \ such that $\big(1+||(u_n,v_n)||_E\big)||I'(u_n,v_n)||_{E'}\to0$, it is necessary to show that $(u_n,v_n)$ is bounded. 
	Seeking a contradiction, suppose \ that \ $||(u_n,v_n)||_E\to +\infty$, \ up \ to \ \ subsequences and define  $(\tilde{u}_n,\tilde{v}_n):=\dfrac{(u_n,v_n)}{||(u_n,v_n)||_E}$, the normalized sequence. Since it is bounded, \ $(\tilde{u}_n,\tilde{v}_n) \rightharpoonup (\tilde{u},\tilde{v})$ in $E$, up to subsequences, and due to the compact embeddings \ $E\hookrightarrow L^s_{h(x)}(\mathbb{R}^N,\mathbb{R}^2)$ for $2\leq s<+\infty$ if $N=2$ and $2\leq s < 2^\#$ if $N>2$, it implies that $(\tilde{u}_n,\tilde{v}_n) \to (\tilde{u},\tilde{v})$ in $L^s_{h(x)}(\mathbb{R}^N,\mathbb{R}^2)$.
	
	\hspace{0.75cm}Suppose \ that \ $(\tilde{u}(x),\tilde{v}(x))=(0,0)$ almost everywhere  in $\mathbb{R}^N$  and  for ${\varepsilon\in(0,1)}$ define
	$U^\varepsilon_n = (t^\varepsilon_nu_n,\varepsilon{t^\varepsilon_n}v_n) \in E$ such that
	\begin{equation}\label{esl5e1}
	I(U^\varepsilon_n) = \max_{t \in [0,1]}I\big(tu_n,\varepsilon{tv_n}\big).
	\end{equation}
	Provided that $(\tilde{u},\tilde{v}) \to 0$ in $L^s_{h(x)}(\mathbb{R}^N,\mathbb{R}^2)$, from $(F'_1)$ for each $s>0$ fixed, there exists $C_s>0$, which does not depend on $\varepsilon$, such that
	\begin{eqnarray}\label{esl5e2}
	&&\int_{\mathbb{R}^N}F\big(x,s\tilde{u}_n(x),\varepsilon{s\tilde{v}_n(x)}\big)\ dx\nonumber\\ &&\leq C_s\left\{||(\tilde{u}_n,\tilde{v}_n)||^2_{L^2_{h(x)}(\mathbb{R}^N,\mathbb{R}^2)} + ||(\tilde{u}_n,\tilde{v}_n)||^p_{L^p_{h(x)}(\mathbb{R}^N,\mathbb{R}^2)}\right\} \to 0,
	\end{eqnarray}
	as $n \to +\infty$. Moreover, since $||(u_n,v_n)||_E\to +\infty$, for $n$ large enough one has
	\begin{eqnarray}\label{esl5e3}
	I(U^\varepsilon_n)&\geq& I\left(\frac{su_n}{||(u_n,v_n)||_E},\frac{\varepsilon sv_n}{||(u_n,v_n)||_E}\right)\nonumber\\
	&=& \dfrac{s^2}{2}\Big(||(\tilde{u}_n,0)||_E^2 - \varepsilon^2||(0,\tilde{v}_n)||_E^2\Big) - \int_{\mathbb{R}^N}F\big(x,s\tilde{u}_n(x),{\varepsilon s\tilde{v}_n(x)}\big)\ dx.
	\end{eqnarray}
	Since $I(u_n,v_n)\geq c-b>0$, it follows that $||(\tilde{u}_n,0)||_E^2>||(0,\tilde{v}_n)||_E^2$. Combining  this inequality with $1 = ||(\tilde{u}_n,0)||_E^2+||(0,\tilde{v}_n)||_E^2$, it implies that $||(\tilde{u}_n,0)||_E^2> \dfrac{1}{2}$. Hence, from (\ref{esl5e2}) and (\ref{esl5e3}) it implies that
	\begin{eqnarray*}
		\liminf_{n \to+\infty}I(U^\varepsilon_n)&\geq& \dfrac{s^2}{2}\liminf_{n \to+\infty}\Big(||(\tilde{u}_n,0)||_E^2 - \varepsilon^2||(\tilde{u}_n,0)||_E^2\Big)\\
		&=& \dfrac{s^2}{2}(1-\varepsilon^2)\liminf_{n \to+\infty}||(\tilde{u}_n,0)||_E^2\\
		&\geq& \dfrac{s^2}{4}(1-\varepsilon^2),
	\end{eqnarray*}
	for all $s>0$ and $\varepsilon \in (0,1)$. Thus,
	\begin{equation}\label{esl5e0}
	\liminf_{n \to+\infty} I(U^\varepsilon_n)=+\infty \quad \text{for \ all} \quad \varepsilon \in (0,1).
	\end{equation}
	\hspace{0.75cm}On the other hand, if there exists a subsequence $(n_k)\subset \mathbb{N}$, such that for some $\varepsilon_{n_k} \in (0,1)$ one has
	$t^{\varepsilon_{n_k}}_{n_k} \in (0,1)$, in view of (\ref{esl5e1}) it implies that $I'(U^{\varepsilon_{n_k}}_{n_k})U^{\varepsilon_{n_k}}_{n_k} = 0$. In addition, since $t^{\varepsilon_{n_k}}_{n_k}\in(0,1)$, in virtue of $(F'_3)$ it follows that
	\begin{eqnarray}\label{esl5e4}
	I(U^{\varepsilon_{n_k}}_{n_k})&=& I(U^{\varepsilon_{n_k}}_{n_k})-\dfrac{1}{2}I'(U^{\varepsilon_{n_k}}_{n_k})U^{\varepsilon_{n_k}}_{n_k}\nonumber\\ &=& \dfrac{1}{2}\int_{\mathbb{R}^N}Q\big(x,t^{\varepsilon_{n_k}}_{n_k}u_{n_k}(x),{\varepsilon_{n_k}} t^{\varepsilon_{n_k}}_{n_k}v_{n_k}(x)\big)\,dx\nonumber\\
	&\leq& \dfrac{D}{2}\int_{\mathbb{R}^N}Q\big(x,u_{n_k}(x),v_{n_k}(x)\big)\,dx\nonumber\\
	&=& \dfrac{D}{2}\Big\{I(u_{n_k},v_{n_k}) - \dfrac{1}{2}I'(u_{n_k},v_{n_k})(u_{n_k},v_{n_k})\Big\}\nonumber\\
	&\leq& M,
	\end{eqnarray}
	for some $M>0$, provided that $(u_{n_k},v_{n_k})$ is also a Cerami sequence. However, from (\ref{esl5e4}) one arrives at
	$$\limsup_{k\to+\infty}I(U^{\varepsilon_{n_k}}_{n_k})\leq M,$$
	contradicting (\ref{esl5e0}). Thus, such a subsequence $(n_k) \subset \mathbb{N}$ cannot exist and hence, there exists $n_0 \in \mathbb{N}$ such that for all $n\geq n_0$ and for all $\varepsilon \in (0,1)$ one has $t^\varepsilon_n=1$. Furthermore, by the continuity of $I$ it yields
	$$I(U^\varepsilon_n) = I(u_n,\varepsilon v_n) \to I(u_n,v_n) \quad \text{as} \quad \varepsilon \to 1,$$
	for each $n \geq n_0$. Then, for each $n \geq n_0$ there exists $\varepsilon_n \in (0,1)$ such that 
	\begin{equation}\label{esl5e01}
	I(U^{\varepsilon}_n) < I(u_n,v_n)+\dfrac{1}{n},
	\end{equation}
	for all $\varepsilon \in [\varepsilon_n,1]$. Hence, if $\tilde{\varepsilon}_n \in [\varepsilon_n,1]$ is such that $I(U^{\tilde{\varepsilon}_n}_n)\geq I(U^{\varepsilon}_n)$ for all $\varepsilon \in [\varepsilon_n,1]$, in view of (\ref{esl5e01}) it follows that
	\begin{equation}\label{esl5e02}
	\limsup_{n\to+\infty}I(U^{\tilde{\varepsilon}_n}_n)\leq\limsup_{n\to+\infty}I(u_n,v_n) = c.
	\end{equation}
	However, from (\ref{esl5e0}) for each $\varepsilon\in[\varepsilon_n,1)$ there exists $n_\varepsilon\in\mathbb{N}$ such that $n\geq n_\varepsilon$ implies $I(U^\varepsilon_n)\geq \dfrac{1}{1-\varepsilon}\geq\dfrac{1}{1-\varepsilon_n}$, then, it yields
	\begin{equation}\label{esl5e03}
	\liminf_{n\to+\infty}I(U^{\tilde{\varepsilon}_n}_n)\geq\liminf_{n\to+\infty}I(U^{\varepsilon}_n) \geq \lim_{n \to +\infty}\dfrac{1}{1-\varepsilon_n} = +\infty,
	\end{equation}
	since $\varepsilon_n \to 1,$ in virtue of (\ref{esl5e01}). Thus, (\ref{esl5e03}) contradicts (\ref{esl5e02}).
	
	\hspace{0.75cm}Therefore, $(\tilde{u},\tilde{v})$ cannot be the null function, namely there exists $\Omega\subset\mathbb{R}^N$ with $|\Omega|>0$, such that $(\tilde{u}(x),\tilde{v}(x))\not= (0,0)$ for all $x\in \Omega$. Furthermore, since
	\begin{eqnarray*}
	\int_{\mathbb{R}^N}\dfrac{F(x,u_n(x),v_n(x))}{||(u_n,v_n)||_E^2}dx |&=& \dfrac{1}{2}\Big(||(\tilde{u}_n,0)||_E^2 - ||(0,\tilde{v}_n)||_E^2\Big)- \dfrac{I(u_n,v_n)}{||(u_n,v_n)||_E^2}\\
	&\leq& \dfrac{1}{2}||(\tilde{u}_n,0)||_E^2\leq \dfrac{1}{2},
\end{eqnarray*}
	one has
	\begin{equation}\label{esl5e5}
	\dfrac{1}{2}\geq \int_{\Omega}\dfrac{F(x,u_n(x),v_n(x))}{|(u_n(x),v_n(x))|^2}|(\tilde{u}_n,\tilde{v}_n)|^2dx.
	\end{equation}
	Nevertheless, in view of $(F'_2)$, for all $x\in\Omega$ it implies that
	$$\dfrac{F(x,u_n(x),v_n(x))}{|(u_n(x),v_n(x))|^2}|(\tilde{u}_n,\tilde{v}_n)|^2 \to +\infty,$$
	as $n \to +\infty$, since $|(u_n(x),v_n(x))|= |(\tilde{u}_n(x),\tilde{v}_n(x))|\ ||(u_n,v_n)||_E \to +\infty$ for all $x \in \Omega$, as $n \to +\infty$. Passing (\ref{esl5e5}) to the limit and applying Fatou's Lemma, it yields
	$$\dfrac{1}{2}\geq \liminf_{n \to+\infty}\int_{\Omega}\dfrac{F(x,u_n(x),v_n(x))}{|(u_n(x),v_n(x))|^2}|(\tilde{u}_n,\tilde{v}_n)|^2dx =+\infty,$$
	which gives a contradiction. Therefore, $(u_n,v_n)$ is bounded, and the result holds.
\end{proof}

\hspace{0.75cm}At long last, the main result is proved.

\begin{proof}[Proof of Theorem \ref{est2}]
	The same proof of Theorem \ref{est1} is applicable.
	\end{proof}

\hspace{0.75cm}It is worth to mention that depending on the nonlinearity, the solutions $(u,v)$ found in Theorem \ref{est1} or in Theorem \ref{est2} could be either semi-trivial  $(u,0)$ or vectorial $(u,v)$, but never semi-trivial $(0,v)$, since such solutions have positive energy $c>0$.

\hspace{0.75cm}We also would like to point out that, instead of asking for hypothesis $(F'_3)$ one, alternatively, could ask for the following assumption so as to get the boundedness of Cerami sequences. \\
	
	\hspace{-0.7cm} \ $(F''_3)$ There \ exist \ constants $c_1>0$ \ and \ $\theta>\min\left\{1,\dfrac{N}{2}(p-2)\right\}$ such that for all ${(x,u,v) \in \mathbb{R}^N\times\mathbb{R}^2}$,
	$$ Q(x,u,v)\geq c_1h(x)|(u,v)|^\theta.$$
	
	\vspace{0.3cm}
	\hspace{0.75cm}Since the functional associated to system (\ref{ES}) is indefinite, hypothesis $(F''_3)$ is relevant to establish the boundedness of Cerami sequences. Indeed, $(F''_3)$ is a consequence of both following conditions $(AR)$ and $(F'_4)$, which also give the boundedness of Cerami sequences. \\
	
	\hspace{-0.4cm}$(AR)$ \ \ There exists $\mu>2$ such that for all $(x,u,v) \in \mathbb{R}^N\times\mathbb{R}^2$,
	
	$$0\leq \mu F(x,u,v)\leq F_u(x,u,v)u + F_v(x,u,v)v;$$
	
	\hspace{-0.4cm}$(F'_4)$ \ \ There \ exist \ constants \ $c_1>0$ and $\theta>\min\left\{1,\dfrac{N}{2}(p-2)\right\}$ such that for all ${(x,u,v) \in \mathbb{R}^N\times\mathbb{R}^2}$,
	$$F(x,u,v)\geq c_1h(x)|u,v|^\theta.$$\\
	Moreover, only hypothesis $(AR)$ is not sufficient for the boundedness, provided that standard arguments involving this assumption only bounds the difference
	$$||(u_n,0)||_E^2-||(0,v_n)||_E^2,$$
	which appears in the functional. 
	Furthermore, considering $\theta > \dfrac{N}{2}\dfrac{\alpha}{\alpha -1}(p-2)$ instead of $\theta>\dfrac{N}{2}(p-2)$ and exploiting the hypotheses assumed on $V_i$, it is possible to weak hypothesis $(F''_3)$ only assuming
	$$ F_u(x,u,v)u + F_v(x,u,v)v - 2F(x,u,v)\geq c_1|(u,v)|^\theta.$$
\hspace{0.75cm}In this way, it is possible to prove the following lemma.
\begin{lemma}	\label{esl42}
	Under the hypotheses $(V_1)_i-(V_3)_i$, $(V)$, $(F'_1)-(F'_2)$ and $(F''_3)$, $I$ satisfies $(I_4)$.	
\end{lemma}

This result is obtained by bounding the Cerami sequences first in $L^\theta_{h(x)}(\mathbb{R}^N,\mathbb{R}^2)$ and after in $E$, making use of the growth of $F$. For the details on this argument, see \cite{So} and \cite{Ra}.

\end{document}